\documentclass[11pt]{amsart}

\usepackage{amsfonts}
\usepackage{amssymb}
\usepackage{amsthm}
\usepackage{amsmath}
\usepackage[usenames,dvipsnames,svgnames,table]{xcolor}
\usepackage[pdfauthor={Lyons, Mann, Tessera, and Tointon},bookmarksnumbered,hyperfigures,pdftex,pagebackref,colorlinks=true,linkcolor=BrickRed,urlcolor=BrickRed,citecolor=OliveGreen,pdfstartview=FitH]{hyperref}
\usepackage{enumitem}
\usepackage[top=3.5cm, bottom=3.5cm, left=2.5cm, right=2.5cm]{geometry}
\usepackage[capitalise]{cleveref}
\usepackage{microtype}
\usepackage[normalem]{ulem}
\usepackage{mathtools}

\newtheorem{prop}{Proposition}[section]
\newtheorem{theorem}[prop]{Theorem}
\newtheorem{lemma}[prop]{Lemma}

\newtheorem{corollary}[prop]{Corollary}

\newtheorem*{theorem*}{Theorem}

\theoremstyle{definition}

\newtheorem{remark}[prop]{Remark}

\theoremstyle{remark}

\newtheorem*{remark*}{Remark}

\newcommand{\N}{\mathbb{N}}

\newcommand{\Z}{\mathbb{Z}}

\newcommand{\E}{\mathbb{E}}

\newcommand{\Prob}{\mathbb{P}}

\newcommand{\eps}{\varepsilon}

\newcommand{\Aut}{\text{\textup{Aut}}\,}

\newcommand*{\cl}{\mathop{\rm{cl}}\nolimits}
\newcommand*{\GL}{\mathop{\rm{GL}}\nolimits}
\newcommand\flr[1]{\lfloor #1 \rfloor}
\newcommand\ceil[1]{\lceil #1 \rceil}
\newcommand{\mg}{\mathsf{mingr}}  
\newcommand\pc{p_{\mathrm c}}
\newcommand\ue{{\mathrm e}}
\newcommand\ud{{\mathrm d}}
\newcommand{\Cay}{\operatorname{Cay}}

\renewcommand{\ge}{\geqslant}
\renewcommand{\geq}{\geqslant}
\renewcommand{\le}{\leqslant}

\newcommand\redsout{\bgroup\markoverwith
   {\textcolor{red}{\rule[0.5ex]{2pt}{1pt}}}\ULon}

\makeatletter
\@namedef{subjclassname@2020}{%
  \textup{2020} Mathematics Subject Classification}
\makeatother

\title[Explicit universal minimal constants for polynomial growth of groups]{Explicit universal minimal constants\\ for polynomial growth of groups}
\date{}
\author{Russell Lyons}
\address{Indiana University, 831 E 3rd St., Bloomington, IN 47405-7106 USA}
\email{rdlyons@indiana.edu}
\author{Avinoam Mann}
\address{Einstein Institute of Mathematics,
Hebrew University, Givat Ram, Jerusalem 91904, Israel}
\email{avinoam.mann@mail.huji.ac.il}
\author{Romain Tessera}
\address{Institut de Math\'ematiques de Jussieu-Paris Rive Gauche, France}
\email{tessera@phare.normalesup.org}
\author{Matthew Tointon}
\address{School of Mathematics, University of Bristol, United Kingdom}
\email{m.tointon@bristol.ac.uk}
\thanks{R.L.\ partially supported by NSF grant DMS-1954086 and the Simons Foundation.}


\numberwithin{equation}{section}

\begin{document}

\begin{abstract}
Shalom and Tao showed that a polynomial upper bound on the size of a single, large enough ball in a Cayley graph implies that the underlying group has a nilpotent subgroup with index and degree of polynomial growth both bounded   effectively. The third and fourth authors proved the optimal bound on the degree of polynomial growth of this subgroup, at the expense of making some other parts of the result ineffective. In the present paper we prove the optimal bound on the degree of polynomial growth without making any losses elsewhere.  As a consequence, we show that there exist explicit positive numbers $\eps_d$ such that in any group with growth at least a polynomial of degree $d$, the growth is at least $\eps_dn^d$. We indicate some applications in probability; in particular, we show that the gap at $1$ for the critical probability for Bernoulli site percolation on a Cayley graph, recently proven to exist by Panagiotis and Severo, is at least $\exp\bigl\{-\exp\bigl\{17 \exp\{100 \cdot 8^{100}\}\bigr\}\bigr\}$.
\end{abstract}

\maketitle

\tableofcontents

\section{Introduction}
We investigate the growth of finitely generated groups. Given a group $G$ that is generated by a finite subset $X$, we let $s_n(G) = s_n(G, X)$ be the number of elements of $G$ that can be expressed as a product of at most $n$ elements from $X \cup X^{-1}$. If for some $n$ we have $s_n(G) \le 2n$, then $G$ is finite. Indeed, if $G$ is infinite, then for all $n \ge 1$, there exists an element $s$ of length $2n$, which we may write as $s = uv$ where $u$ and $v$ each have length $n$. Then $u \ne v^{-1}$, so that $s_n(G) - s_{n-1}(G) \ge 2$ and $s_n(G) \ge 2n+1$. This inequality is best possible, as both $\Z$ and $(\Z/2\Z) * (\Z/2\Z)$ (with their standard generators) have $s_n(G) = 2n+1$ for all $n$.

Wilkie and van den Dries \cite{WvdD} showed that if $G$ is infinite and the inequality $s_n(G) < (n+1)(n+2)/2$ holds for some $n$, then $G$ is virtually cyclic, and (hence) has linear growth. In fact, they showed that if $s_n(G) - s_{n-1}(G) \le n$ for some $n \ge 1$, then $G$ has a cyclic subgroup of index at most $\bigl(s_n(G) - s_{n-1}(G)\bigr)^4/2$. The bound on the index was improved by Imrich and Seifter \cite{IS} to $s_n(G) - s_{n-1}(G)$, which is sharp.

Results of this type are known for higher rates of growth.
If there exist numbers $C$ and $d$ such that $s_n(G) \le Cn^d$ for all $n$, then $G$ is said to be of {\it polynomial growth}. In that case, the {\it growth degree} $\deg(G)$ of $G$ is the infimum of the numbers $d$ for which another number $C$ can be found such that the inequality above is satisfied. This degree is independent of the generator system $X$, and can be characterized equivalently by $\deg(G) := \limsup \frac{\log s_n(G)}{\log n}$. If $G$ does not have polynomial growth, then, given any numbers $C$ and $d$, the inequality $s_n(G) > Cn^d$ holds for infinitely many $n$. In other words, the upper limit above is infinite.

If $G$ is nilpotent of {\it class\/} $\cl(G) = c$ with lower central series  $G = \gamma_1(G) \vartriangleright \gamma_2(G) \vartriangleright \cdots \vartriangleright \gamma_c(G) \vartriangleright \gamma_{c+1}(G) = \{1\}$, then, as Bass \cite{bass} and Guivarc'h \cite{guiv} showed, the growth degree can be expressed as $r := \sum_{i=1}^{c} ir(i)$, where $r(i)$ is the torsion-free rank of $\gamma_i(G)/\gamma_{i+1}(G)$, i.e., the number of infinite factors in the decomposition of this quotient as a direct sum of cyclic groups. The \emph{Hirsch length} $h(G)$ of $G$ is defined to be $\sum_{i=1}^c r(i)$; obviously $h(G) \le  r \le h(G) \cdot c$. A virtually nilpotent group has the same growth degree as its nilpotent, finite-index subgroups. The above formula shows that the degree is an integer. Given a group $G$ with a finite-index, nilpotent subgroup, $H$, we define the \emph{Hirsch length} $h(G)$ of $G$ to be $h(H)$.

A celebrated theorem of Gromov \cite{gromov} established a conjecture of Milnor that a finitely generated group $G$ has polynomial growth (if and) only if $G$ is virtually nilpotent. Building on work of Kleiner \cite{kleiner}, Shalom and Tao \cite{st} subsequently gave a finitary version of this statement, showing that a polynomial upper bound on the size of just a single ball (of large enough radius) implies that a group is virtually nilpotent. Their result gives effective bounds on both the index and the degree of polynomial growth of the nilpotent subgroup, and on how large the radius needs to be in order for the theorem to hold. In relatively recent work, the third and fourth authors made the bound on the degree of polynomial growth optimal at the expense of some effectiveness elsewhere. The main aim of the present work is to obtain the optimal bound on the degree of polynomial growth without sacrificing effectiveness elsewhere. We also present some applications to probability.

Shalom and Tao's refinement of Gromov's theorem is the following.
\begin{theorem}[Shalom--Tao {\cite[Theorem 1.8]{st}}]\label{thm:st}
There exists an absolute constant $C$
such that if $G$ is a group with finite generating set $X$, and if $s_n(G, X) \le n^d$ for some $d \ge 1$ and some integer $n\ge\exp(\exp(Cd^C))$, then $G$ has a nilpotent subgroup of index $O_{n,d}(1)$
and Hirsch length and class at most $C^d$,
whence $\deg(G)\le C^{2d}$.
\end{theorem}
Here and elsewhere, we adopt the notational convention that if $X$ is a real quantity and $z_1,\ldots,z_k$ are parameters, then $O_{z_1,\ldots,z_k}(X)$ denotes a quantity that is at most a constant multiple of $X$, with the constant depending only on the parameters $z_1,\ldots,z_k$.

\cref{thm:st} says that a polynomial upper bound on the size of a single, large enough ball is enough to imply that a group is virtually nilpotent, and to give some quantitative control over the complexity of the virtual nilpotency. A bound on $C$ can be computed explicitly from the proof; the authors assert that one such bound should be $100$. The bound $O_{n,d}(1)$ on the index could in principle be made effective, but the authors instead use an ineffective compactness argument, saying that the corresponding effective argument would be `substantially longer' and result in a bound of Ackermann type in $d$.

\begin{remark*}In his original paper, Gromov applied a compactness argument together with his own theorem to obtain a similar conclusion to \cref{thm:st} \cite[\S8]{gromov}. This yields ineffective bounds and requires the stronger hypothesis that $|s_n(G)|\le n^d$ for some $d \ge 1$ and all $n=2,\ldots,n_0$, for some $n_0=n_0(d)$.
\end{remark*}

Given the polynomial of degree $d$ appearing in the hypothesis of \cref{thm:st}, it is natural to wonder whether $\deg(G)$ should also be at most $d$. This amounts to asking whether a group can grow like a polynomial of degree $d$ at small scales and then accelerate to grow like a polynomial of higher degree at large scales. It turns out that if one considers instead a `relative' condition of the form $|s_n(G)|\le Cn^d|s_1(G)|$, then this can indeed occur (see \cite[Example 1.11]{tao} for details). However, the third and fourth authors showed that this does not occur in the context of \cref{thm:st} by proving the following result, which verified a conjecture of Benjamini. We write $\N$ for the set of strictly positive integers.
\begin{theorem}[{\cite[Theorem 1.11]{tt.proper.progs}}]\label{thm:tt.orig}
For every $d\in\N$, there exists $\eps_d>0$ such that if $G$ is a group with finite generating set $X$ and if $s_n(G, X) < \eps_dn^d$ for some $n\in\N$, then $s_m(G, X)\le O_d\bigl((m/n)^{d-1}s_n(G, X)\bigr)$ for every integer $m\ge n$.
\end{theorem}
\cref{thm:tt.orig} relies on Breuillard, Green, and Tao's structure theorem for approximate groups \cite{bgt}, and as such does not give an effective computation of $\eps_d$. The bound $O_d\bigl((m/n)^{d-1}s_n(G, X)\bigr)$ is also ineffective in the original reference for \cref{thm:tt.orig}, but in forthcoming work, the third and fourth authors will give an improved proof of \cref{thm:tt.orig} that results in an effective bound.

As an immediate consequence of \cref{thm:st,thm:tt.orig}, we obtain the optimal bound on $\deg(G)$ in the Shalom--Tao theorem, as follows.
\begin{corollary}\label{thm:tt}
For every $d\in\N$, there exists $\eps_d>0$ such that if $G$ is a group with finite generating set $X$, and if $s_n(G, X) < \eps_dn^d$ for some $n\in\N$, then $G$ has a nilpotent subgroup of index $O_{n,d}(1)$, and $\deg(G)\le d-1$.
\end{corollary}
Note that, although the hypothesis $s_n(G, X) < \eps_dn^d$ in this result might at first glance appear rather stronger than the hypothesis $s_n(G, X) \le n^d$ of \cref{thm:st}, provided $n>1/\eps_{d+1}$, the latter bound implies the former with $d+1$ in place of $d$.

It appears to be beyond the reach of current methods to give an explicit value of $\eps_d$ in \cref{thm:tt.orig}. Nonetheless, in the present work we obtain \cref{thm:tt} directly and elementarily from \cref{thm:st}, bypassing the Breuillard--Green--Tao theorem completely and making $\eps_d$ effective in \cref{thm:tt} without any losses elsewhere. This leads in turn to effective constants $\eps_d$ in the following trivial consequence of \cref{thm:tt}.
\begin{corollary}\label{cor:min.growth}
Let $d\in\N$, and suppose that $G$ is a group satisfying $\deg(G)\ge d$ and $X$ is a finite generating set for $G$. Then $s_n(G, X)\ge\eps_dn^d$ for every $n\in\N$, where $\eps_d>0$ is the constant given by \cref{thm:tt}.
\end{corollary}
This has particular relevance to the study of probability on groups, where lower bounds on growth have numerous applications.

\subsection*{Main new results} Our first main result deals with groups of growth exactly $d$, and for that reason it does not rely on the Shalom--Tao theorem.

\begin{theorem}\label{thm:v.nilp.min.growth}
Let $d\in\N$, and suppose $G$ is a virtually nilpotent group with polynomial growth of degree $d$. Let $X$ be a finite generating set for $G$. Then
\[
s_n(G,X)\ge\frac{n^d}{2^{d(d+2)}g\bigl(h(G)\bigr)^d}
\ge\frac{n^d}{2^{d(d+2)}g(d)^d}
\]
for every $n\in\N$, where $g(k)$ is the maximum order of a finite subgroup of $\GL_k(\Z)$.
\end{theorem}
An upper bound for $g(k)$ was given already by Minkowski \cite{mink} in 1887. One such bound is
\begin{equation}\label{eq:g(d)}
g(k)\le(2k)!
\end{equation}
(see equation (16) on p.~175 of \cite{newman}). See also \cite{feit} and the remarks about $g(k)$ on pp.~88--89 of \cite{mann.book}.

Combining Theorem \ref{thm:st} with Theorem \ref{thm:v.nilp.min.growth}, we deduce an effective version of \cref{thm:tt}, as follows.

\begin{theorem}\label{cor:effective.min.growth}
We may take
\begin{equation} \label{eq:epsd.def}
\eps_d=\min\left\{\frac{1}{2^{3C^{4d}}g(C^{d})^{C^{2d}}},\frac{1}{\exp(d\exp(Cd^C))}\right\}
\end{equation}
in \cref{thm:tt}, and hence also in \cref{cor:min.growth}. Moreover, this yields the same bound on the index of the nilpotent subgroup as \cref{thm:st}.
\end{theorem}

The second term
in the expression of $\eps_d$ is directly related to the lower bound on $n$
in the Shalom--Tao theorem. We observe that the second term is
asymptotically smaller than the first one (after taking logs of the reciprocals twice, the first one becomes $\simeq d$, while the
second one becomes $\simeq d^C$).

\begin{remark}
Define $\mg(d) := \inf \{ s_n(G, X)n^{-d} \}$, where the infimum is taken over all $n \in \N$ and all virtually nilpotent groups $G$ with polynomial growth of degree $d$ and generating sets $X$. Obviously $\mg(1) = 2$. We do not know the values of $\mg(d)$ for other $d$. \Cref{thm:v.nilp.min.growth} gives a lower bound on $\mg(d)$. 
For an upper bound, note that when $\Z^d$ is generated by $d$ elements, we have $\mg(d) \le \lim_{n \to\infty} s_n(\Z^d)/n^d = 2^d/d!$: up to terms of order $n^{d-1}$, $s_n(G)$ is the volume of a hyperoctahedron, which, in turn, is $2^d$ times the volume of its intersection with the nonnegative orthant. We can do better, however: For $d \ge 2$, consider the affine Coxeter group $\widetilde B_d$, which has growth degree $d$ and so-called exponents $1, 3, \ldots, 2d-1$; see \cite[Appendix A1]{bb}. By a formula of Bott (see \cite[Theorem 7.1.10]{bb} or \cite[Theorem 3.8]{stein}), we have for the Coxeter generators, 
\[
\sum_{n \ge 0} s_n(\widetilde B_d) z^n 
= 
\frac1{(1-z)^{d+1}} \prod_{k=1}^d \frac{1 - z^{2k}}{1 - z^{2k-1}}
\]
for $|z| < 1$.
By \cite[Proposition 51]{pansu} and \cite[Lemma 3.2]{stoll}, we have for every group $G$ of polynomial growth degree $d$ that
\[
\lim_{n \to\infty} s_n(G) n^{-d}
=
\frac1{d!} \lim_{z \uparrow 1} (1-z)^{d+1} \sum_{n \ge 0} s_n(G) z^n,
\]
whence
\[
\mg(d)
\le
\lim_{n \to\infty} s_n(\widetilde B_d) n^{-d}
=
\frac1{d!} \frac{(2d)!!}{(2d-1)!!}.
\]
Note, in particular, that ${(2d)!!}/{(2d-1)!!} \sim \sqrt{\pi d}$ as $d \to\infty$.
In fact, there are a few other Coxeter groups that give still better bounds: $\mg(2) \le \lim_{n \to\infty} s_n(\widetilde G_2) n^{-2} = (12/5)/2!$, $\mg(6) \le \lim_{n \to\infty} s_n(\widetilde E_6) n^{-6} = (324/77)/6!$, $\mg(7) \le \lim_{n \to\infty} s_n(\widetilde E_7) n^{-7} = (9216/2431)/7!$, and $\mg(8) \le \lim_{n \to\infty} s_n(\widetilde E_8) n^{-8} = (99532800/30808063)/8!$; again, see \cite[Appendix A1]{bb} for the definitions and exponents of these groups.
\end{remark}

\begin{remark*}The situation for groups of exponential growth is known to be quite different from the situation for groups of polynomial growth described by \cref{cor:min.growth}. On the one hand, some classes of groups are known to have uniformly exponential growth over all generating sets, in the sense that there is a constant $c > 1$ depending only on the group such that the ball of radius $n$ with respect to an arbitrary generating set has at least $c^n$ elements; indeed, the same $c > 1$ sometimes exists even for an entire class of groups. On the other hand, it is known that there are groups of exponential growth whose rate of growth on the exponential scale is arbitrarily small for certain sets of generators. See, e.g., \cite{bucher-talabutsa} for results and history of exponential growth. There is much less knowledge for groups of intermediate growth: it is not even known whether there are such groups whose balls of radius $n$ have asymptotically fewer than $e^{c \sqrt n}$ elements.
\end{remark*}

\subsection*{Vertex-transitive graphs}
Trofimov \cite[Theorem 2]{trof} famously extended Gromov's theorem to vertex-transitive graphs of polynomial growth, showing that any such graph has a quotient that looks roughly like a virtually nilpotent Cayley graph in a certain precise sense. Woess \cite[Theorem 1]{woess} subsequently gave a simple proof of this result using the theory of topological groups. Inspired by Woess's proof, and applying a version of the Breuillard--Green--Tao theorem for locally compact groups due to Carolino \cite{Ca}, the third and fourth authors of the present work gave a finitary version of Trofimov's theorem that allowed them to extend \cref{thm:tt.orig} to vertex-transitive graphs \cite[Corollary 1.5]{tt.trof}.

Unfortunately, we are not aware of an effective result for locally compact groups that could be used to bypass Carolino's result in the same way that we use Shalom and Tao's result to bypass the Breuillard--Green--Tao theorem in our proof of \cref{thm:tt}. Nonetheless, using Trofimov's result we can at least obtain the following generalisation of \cref{thm:v.nilp.min.growth},
in which we write $s_n(\Gamma)$ for the number of vertices inside a ball of radius $n$ in a vertex-transitive graph $\Gamma$.

\begin{corollary}\label{cor:VT.min.growth}\label{cor:VT.min.growth.detailed}
Let $d\in\N$, and suppose $\Gamma$ is a vertex-transitive graph with polynomial growth of degree exactly $d$. Then
\[
s_n(\Gamma)\ge\frac{n^d}{2^{d(d+2)}g(d)^{d+1}}
\]
for every $n\in\N$.
\end{corollary}
See also \cref{cor:VT.min.almostEffective} for a partially effective version of \cref{cor:VT.min.growth} valid in a vertex-transitive graph of growth degree \emph{at least} $d$.

\subsection*{Minimal polynomial-growth constants and probability} Results such as \cref{cor:min.growth} can be used to give universal bounds on various quantities in probability. For example, given a vertex-transitive graph $\Gamma$ with vertex set $V$, edge set $E$, and valency $\Delta$,
define \emph{lazy simple random walk} on $\Gamma$ to be the Markov chain whose transition probabilities from $y \in V$ to $z \in V$ are
\[
p(y, z) = \begin{cases} 1/(2\Delta) &\text{if $\{y, z\} \in E$,}\\
                 1/2 &\text{if $y = z$,}\\
                 0 &\text{otherwise.}\end{cases}
\]
Write $p_t(y, z)$ for the $t$-step transition probabilities. A special case of \cite[Corollary 6.6]{LOG} states that if $c, d > 0$ are such that $s_n(\Gamma) \ge c n^d$ for all $n \in \N$, then for all $y, z \in V$ and $t\in\N$, we have
\[
p_t(y, z) \le p_t(y, y) \le \frac{8 d^{(d+5)/2} \Delta^{d/2}}{c \ue^{d/2}}t^{-d/2}.
\]

Combining this with our results yields several corollaries, such as the following.
\begin{corollary}\label{cor:return.prob} Let $d\in\N$, and suppose that $\Gamma$ is a Cayley graph of a group having growth degree at least $d$ or is a vertex-transitive graph with polynomial growth of degree exactly $d$. Then for every $y, z \in G$ and $t \ge 1$ we have
\[
p_t(y, z) \le p_t(y, y) \le \frac{8 d^{(d+5)/2} \Delta^{d/2}}{\eps_d\ue^{d/2}}t^{-d/2},
\]
where $\Delta$ is the valency
and $\eps_d>0$ is the constant given by \eqref{eq:epsd.def} in the case of a Cayley graph and is
$1/\bigl(2^{d(d+2)}g(d)^{d+1}\bigr)$
otherwise.
\end{corollary}

\cref{cor:return.prob} in turn leads to bounds on various other quantities. For example, Panagiotis and Severo \cite{pan-sev} recently showed that there exists a \emph{gap} at $1$ for the critical probability $\pc$ for Bernoulli site (and hence bond \cite[Proposition 7.10]{ly-per}) percolation on a Cayley graph, in the sense that there exists $\eps>0$ such that either $\pc\le 1-\eps$ or $\pc=1$ for every Cayley graph. Using \cref{cor:return.prob} in their argument instead of their bound \cite[(3.4)]{pan-sev} on $p_t(x, y)$ allows one to give an explicit value for $\eps$, as follows.
\begin{corollary}\label{cor:gap}
Let $\Gamma$ be a Cayley graph. Then the critical probability $\pc(\Gamma)$ for Bernoulli site percolation on $\Gamma$ satisfies either
\[
\pc(\Gamma)<1-\exp\bigl\{-\exp\bigl\{17 \exp\{100 \cdot 8^{100}\}\bigr\}\bigr\} =: p_0
\]
or $\pc(\Gamma)=1$. Furthermore, the probability that the identity element belongs to an infinite cluster at value $p_0$ satisfies 
\[
\mathbb{P}_{p_0}[o \leftrightarrow \infty]
>
\exp\bigl\{-9 \exp\{100 \cdot 8^{100}\}\bigr\}.
\]
The same inequalities hold for Bernoulli bond percolation. 
\end{corollary}
See \Cref{appendix} for more details. The Cayley graph with the largest value known of $\pc$ less than 1 is apparently that of the presentation $\langle a, b \mid  a^3, b^2, (ab)^6 \rangle$, which is the 3-12 lattice; there, we have $\pc = \sqrt{1 - 2\sin(\pi/18)} = 0.8079^+$ for site percolation (\cite[p.~278]{sudziff} gives a simple reduction to site percolation on the Kagom\'e lattice, which is the line graph of the hexagonal lattice, whence site percolation on the former is equivalent to bond percolation on the latter, whose critical probability was rigorously determined by \cite{wierman-h}).

The third and fourth authors \cite{tt.resist} have shown that there is a gap at $0$ for escape probabilities of random walks on vertex-transitive graphs, in the sense that there exists an absolute constant $c>0$ such that simple random walk on an arbitrary vertex-transitive graph is either recurrent or has escape probability at least $c$. This constant $c$ is independent of the valency but is not explicit. The results of the present paper allow us, in the special case of Cayley graphs, to replace this non-explicit constant $c$ with an explicit function of the valency $\Delta$. The most immediate such bound follows from noting that the escape probability is equal to $1/\sum_{t \ge 0} p_t(x, x)$, and that a transient Cayley graph has growth degree at least $3$; Corollary \ref{cor:return.prob} then immediately yields a lower bound on the escape probability of the form $K\Delta^{-3/2}$ for a transient Cayley graph, where $K$ is an explicitly computable absolute constant. We can do even better, however, if we pass via an isoperimetric inequality. By an \emph{isoperimetric inequality} in a group $G$ with finite generating set $X$, we mean a lower bound on the size of the \emph{vertex boundary} $\partial A$ of a finite set $A\subseteq G$, defined as $\partial A=A(X\cup X^{-1})\setminus A$. It follows from a well-known result of Coulhon and Saloff-Coste \cite{csc}, with bounds as given by \cite[Theorem 6.29]{ly-per}, that for each positive integer $d$, we have
\[
|\partial A|\ge\frac{|A|}{2\lceil(2|A|/\eps_d)^{1/d}\rceil}
\]
for an arbitrary, non-empty, finite subset $A$ of a group $G$ satisfying $\deg(G)\ge d$, where $\eps_d$ is the quantity appearing in \cref{thm:tt,cor:min.growth}. In particular, this implies the explicit \emph{$d$-dimensional isoperimetric inequality}
\begin{equation}\label{eq:iso}
|\partial A|\ge \frac{\eps_d^{1/d}}8|A|^\frac{d-1}d
\end{equation}
for any such $A$ and $G$.  Inserting \eqref{eq:iso} into the argument of \cite{tt.resist}, one can improve the lower bound $K\Delta^{-3/2}$ on the escape probability described above to $J\Delta^{-2/3}$, where $J$ is an explicitly computable absolute constant. Indeed, this leads to a lower bound on the escape probability of the form $J_d\Delta^{-2/d}$ for any group $G$ satisfying $\deg(G)\ge d$, where $J_d$ is an explicit function of $d$.

For one final example of an application of our results, \cite[Proposition 2.8]{LPS} shows that for every transitive graph, $\E[K_i] \le \sum_{t=0}^\infty (t+1) p_t(x, x)/2$, where $i \ge 0$ and $K_i$ is the number of times $t$ such that the loop-erasure of the (nonlazy) simple random walk path up to time $t$ has exactly $i$ edges (although $p_t(x, x)$ still refers here to the lazy simple random walk). In the case of a Cayley graph of growth degree at least $5$, it follows from \Cref{cor:return.prob} that $\E[K_i] \le 5131 \Delta^{5/2}/\eps_5$. 
An interesting question is whether the dependence on the valency is necessary for bounding $\E[K_i]$.

\section{Background on nilpotent groups}
In this section we present some standard definitions and results from the theory of nilpotent groups. Recall that the set of elements of finite order in a nilpotent group $G$ is a subgroup $T$, called the \emph{torsion subgroup}. If $G$ is generated by a finite set $X$, then $T$ is finite, and the quotient $G/T$ is torsion-free \cite[5.2.7]{robinson} with $s_n(G, X) \ge s_n(G/T, XT)$. In this case, the growth of $G$ is trivially of the same degree as the growth of $G/T$, meaning that in many of our arguments we may assume without loss of generality that any nilpotent groups are torsion-free.

Given elements $g$ and $h$ of a group $G$, we denote by $[g,h]$ the \emph{commutator} $g^{-1}h^{-1}gh$ of $g$ and $h$. More generally, given elements $x_1,\ldots,x_k$ of a group $G$, we define the \emph{simple commutator} $[x_1,\ldots,x_k]$ of \emph{weight} $k$ recursively by $[x_1]=x_1$ and $[x_1,\ldots,x_k]=[[x_1,\ldots,x_{k-1}],x_k]$. By definition, $\gamma_k(G)$ is the subgroup of $G$ generated by the simple commutators of weight $k$ in elements of $G$.

Write $\lambda(k)$ for the \emph{length} of the simple commutator of weight $k$ as an unreduced word in the elements $x_i^{\pm1}$; thus, for example, $\lambda(3)=10$ because $[x_1,x_2,x_3]=x_2^{-1}x_1^{-1}x_2x_1x_3^{-1}x_1^{-1}x_2^{-1}x_1x_2x_3$. It is clear that $\lambda(k+1) = 2\lambda(k) + 2$, whence $\lambda(k) = 3\cdot 2^{k-1} - 2$. We will use only the following consequence:
\begin{equation}\label{eq:l(k)}
\lambda(k)\le2^{k-1}k.
\end{equation}

\begin{lemma}[{\cite[Theorem 10.2.3]{hall} or \cite[Proposition 5.2.6]{tointon.book}}]\label{lem:G_s=<comms>}
Let $G$ be a group with generating set $X$ and let $k\in\N$. Then $\gamma_k(G)/\gamma_{k+1}(G)$ is generated by the image in $G/\gamma_{k+1}(G)$ of the set $\{[x_1,\ldots,x_k]:x_1,\ldots,x_k\in X\}$.
\end{lemma}

\begin{lemma}[{\cite[Lemma 5.5.3 \& Proposition 5.2.7]{tointon.book}}]\label{lem:comms.bilin}
Let $G$ be a group, let $g\in G$, and let $k\in\N$. Then the map
\[
\begin{array}{ccc}
\gamma_k(G)&\to&\gamma_{k+1}(G)/\gamma_{k+2}(G)\\[3pt]
x&\mapsto&[x,g]\gamma_{k+2}(G)
\end{array}
\]
is a homomorphism, the kernel of which contains $\gamma_{k+1}(G)$.
\end{lemma}

\begin{lemma}\label{lem:ab>2} Let $c\in\N$, and let $G$ be a torsion-free nilpotent group of class $c$. For each $i=1,\ldots,c$, write $r(i)$ for the torsion-free rank of $\gamma_i(G)/\gamma_{i+1}(G)$. Then $r(i) \ge 1$ for $1 \le i \le c$, and if $G$ is not cyclic, then $r(1) \ge 2$.
\end{lemma}

\begin{proof}Suppose that $r(k) = 0$ for some $k\in\{1,\ldots,c\}$, and let $k$ be the maximum such. If $k = c$, then $\gamma_c(G)$ is finite, hence trivial, contrary to the definition of $c$. If $k<c$, then all simple commutators of weight $k$ have finite order modulo $\gamma_{k+1}$. \cref{lem:comms.bilin} therefore implies that all simple commutators of weight $k+1$ have finite order modulo $\gamma_{k+2}$. This implies that $r(k+1)=0$, contradicting the maximality of $k$. This establishes our claim that $r(i) \ge 1$ for $1 \le i \le c$.

Now suppose that $r(1) = 1$. Then we can choose a generating set $X$ for $G$ such that only one of the $x_i$ has infinite order modulo $\gamma_2(G)$ (indeed, $X$ generates $G$ if and only if the image of $X$ in $G/\gamma_2(G)$ generates $G/\gamma_2(G)$ \cite[Corollary 10.3.3]{hall}). \cref{lem:comms.bilin} therefore implies that every commutator $[x,y]$ with $x,y\in G$ has finite order in $\gamma_2(G)/\gamma_3(G)$, so that $r(2)=0$. By the first part of the lemma, this implies that $c=1$, so that $G$ is free abelian of rank $1$, i.e., infinite cyclic.
\end{proof}

\begin{corollary}\label{lem:class.vs.deg}
Let $d\ge2$ be an integer and suppose $G$ is a torsion-free nilpotent group with growth degree $d$. Then $c=\cl(G)$ satisfies $c(c+1)\le 2d-2$.
\end{corollary}
\begin{proof}
\cref{lem:ab>2} implies that $d \ge 1 + \sum_{i=1}^c i = 1 + c(c + 1)/2$.
\end{proof}

\begin{lemma}[{\cite[Lemma 5.5.2]{tointon.book}}]\label{lem:comms.linear}
Let $G$ be a group and let $k\in\N$. Then the map
\[
\begin{array}{ccc}
G^k&\to&\gamma_k(G)\\
(x_1,\ldots,x_k)&\mapsto&[x_1,\ldots,x_k]
\end{array}
\]
is a homomorphism in each variable modulo $\gamma_{k+1}(G)$.
\end{lemma}

\section{Minimal polynomial-growth constants for virtually nilpotent groups}

We start by considering the special case of a group that is actually nilpotent, rather than merely virtually nilpotent.

\begin{prop}\label{prop:nilp.min.growth}
Let $d\in\N$, and suppose $G$ is a nilpotent group with polynomial growth of degree $d$. Let $X$ be a finite generating set for $G$. Then
\[
s_n(G,X)\ge\frac{n^d}{2^{d^2}}
\]
for every $n\in\N$.
\end{prop}

The proof of \cref{prop:nilp.min.growth} is by induction on $d$, and we carry out the induction step by examining a certain quotient of $G$ with lower growth degree. We will use the following technical lemma that allows us to compare the growth of $G$ to the growth of this quotient.
Recall that $B_n(G,X)$ denotes the ball of radius $n$ with respect to $X$ centered at the identity element in $G$.

\begin{lemma}\label{lem:lb.quotient.kernel}
Let $G$ be a group with finite generating set $X$, and suppose $H\trianglelefteq G$ is a normal subgroup. Then for every $m,n\ge0$, we have $s_{m+n}(G,X)\ge s_m(G/H,XH/H)\cdot|B_n(G,X)\cap H|$.
\end{lemma}
\begin{proof}
The ball of radius $m$ in $G$ contains a set $A$ of cardinality $s_m(G/H,XH/H)$ with each element belonging to a distinct coset of $H$. The products $ax$ with $a\in A$ and $x\in B_n(G,X)\cap H$ are then distinct elements of the ball of radius $m+n$.
\end{proof}

In the expression $n/2^c c^2$ below and others like this,
we write $x/yz$ to mean $x/(yz)$.

\begin{proof}[Proof of \cref{prop:nilp.min.growth}]
On passing to the quotient of $G$ by its torsion subgroup, we may assume that $G$ is torsion-free.
If $n < 2^{d}$, then $n^d < 2^{d^2}$, whence $s_n(G,X)\ge1 > n^d/2^{d^2}$ and the proposition is satisfied. We may therefore assume that $n\ge2^{d}$.

If $G$ is abelian, then every generating set contains $d$ independent elements that generate a free abelian subgroup $H$ of rank $d$, hence $s_n(G) \ge s_n(H) > n^d/d!$: To see this lower bound, consider only the part of the ball with all coordinates strictly positive.  For integers $x_i > 0$ with $\sum_{i=1}^d x_i \le n$, let $C_x$ be the unit cube $\prod_{i=1}^d (x_i - 1, x_i]$, where $x = (x_1, \ldots, x_d)$. These cubes are disjoint. Suppose $z = (z_1, \ldots, z_d)$ is a real point in the pyramid where $z_i > 0$ for all $i$ and $\sum_{i=1}^d z_i \le n - d$. Then $z$ lies in the cube $C_w$, where $w := (\ceil{z_1}, \ldots, \ceil{z_d})$. Clearly $\sum_{i=1}^d \ceil{z_i} \le n$.  Therefore, the number of such $x$ is at least the volume of this pyramid, which is $(n-d)^d/d! \ge (n/2)^d/d!$. Considering all elements of the ball of radius $n$ with no coordinates equal to $0$ gives the claimed lower bound, $n^d/d!$.
Since $d! < 2^{d^2}$, the proposition holds when $G$ is abelian.


We now prove the proposition by induction on $d$.
The base case, $d = 1$, follows because the only torsion-free such group is the infinite cyclic group, which is abelian.

We now assume that $G$ is nonabelian.

Write $c=\cl(G)$. Because $G$ is nonabelian, $c \ge 2$,
so that $c + 2\log_2 c \le 1 + c(c+1)/2 \le d$ in light of \cref{lem:class.vs.deg}, whence $2^{d} \ge 2^c c^2$.

By \cref{lem:G_s=<comms>}, there exist elements $x_1,\ldots,x_c\in X$ such that $[x_1,\ldots,x_c]\ne1$. Set $H:=\langle[x_1,\ldots,x_c]\rangle$. Given $n\in\N$, we claim first that
\begin{equation}\label{eq:S.cap.H}
|B_{\lfloor n/2\rfloor}(G,X)\cap H|\ge\frac{n^c}{2^{c(c+1)}c^{2c}}.
\end{equation}
Given $L\in\N$, for every integer $k=1,\ldots,L^c$ there exist $m\le c$ and integers $\ell_{11},\ldots,\ell_{1c},\ldots,$ $\ell_{m1},\ldots,\ell_{mc}\in[1,L]$
such that 
$k=\sum_{i=1}^m\prod_{j=1}^c\ell_{ij}$, as we can see by writing $k$ in base $L$. 
\cref{lem:comms.linear} therefore implies that for every such $k$ we have
\[
[x_1,\ldots,x_c]^k=[x_1^{\ell_{11}},\ldots,x_c^{\ell_{1c}}]\cdots[x_1^{\ell_{m1}},\ldots,x_c^{\ell_{mc}}]\in B_{c\lambda(c)L}(G,X)\cap H,
\]
so that $|B_{c\lambda(c)L}(G,X)\cap H|\ge L^c$. Setting $L:=\lfloor n/2^cc^2\rfloor$ and noting that $c\lambda(c)L\le n/2$ by \eqref{eq:l(k)}, we deduce that $|B_{\lfloor n/2\rfloor}(G,X)\cap H|\ge\lfloor n/2^cc^2\rfloor^c$. Since $n\ge2^{d}\ge 2^cc^2$, we have $\lfloor n/2^cc^2\rfloor\ge n/2^{c+1}c^2$, so this proves \eqref{eq:S.cap.H} as claimed.

The degree of polynomial growth of $G/H$ is $d-c < d$, so by induction we may assume that
\[
s_{\lceil n/2\rceil}(G/H,XH/H)\ge\frac{(n/2)^{d-c}}{2^{(d-c)^2}}
=
\frac{n^{d-c}}{2^{(d-c)^2 + d-c}}.
\]
Combining this with \eqref{eq:S.cap.H} and \cref{lem:lb.quotient.kernel},
we deduce that
\[
s_n(G,X)\ge
\frac{n^d}{2^{(d-c)^2 + d-c+ c(c+1)}c^{2c}}
=
\frac{n^d}{2^{(d-c)^2 + d+ c^2 + 2c \log_2 c}}.
\]
It remains to show that $(d-c)^2 + d+c^2 + 2c \log_2 c \le d^2$, in other words, that 
\[
2c(c + \log_2 c) \le (2c - 1)d.
\]
Now
\[
1 + \frac1{c-1}
\le
2\log_2 c
\]
because $c\ge2$. Multiply both sides by $c-1$, add $2c \log_2 c - c + 2c^2$ to both sides, factor the right-hand side, and use the inequality $c + 2 \log_2 c \le d$ established above to get the desired result.
\end{proof}

We now move on to the proof of the more general \cref{thm:v.nilp.min.growth}, writing $g(k)$ from now on for the maximum order of a finite subgroup of $\GL_k(\Z)$, as in that theorem. It is not too difficult to deduce from \cref{prop:nilp.min.growth} a version of \cref{thm:v.nilp.min.growth} in which the lower bound on $s_n(G,X)$ has some dependence on the index of a nilpotent subgroup. The key to removing this dependence is the following result, which is essentially \cite[Theorem 9.8]{mann.book}.

\begin{prop}\label{prop:index.bound}
Suppose that $G$ is a finitely generated virtually nilpotent group. Then there exist normal subgroups $H,N\trianglelefteq G$ with $H\le N$ finite and $[G:N]\le g\bigl(h(G)\bigr)$ such that $N/H$ is torsion-free nilpotent.
\end{prop}
\begin{proof}
This is almost given by \cite[Theorem 9.8]{mann.book}, which says that there exist normal subgroups $H_0,N\trianglelefteq G$ with $H_0\le N$ finite and $[G:N]\le g\bigl(\deg(G)\bigr)$ such that $N/H_0$ is nilpotent. The stronger bound $[G:N]\le g\bigl(h(G)\bigr)$ claimed here can be read directly out of the proof of \cite[Theorem 9.8]{mann.book}, but $N/H_0$ may still not necessarily be torsion-free. Nonetheless, being of finite index in $G$, the subgroup $N$ is also finitely generated \cite[1.6.11]{robinson}, so the torsion subgroup of $N/H_0$ is finite. This subgroup is characteristic in $N/H_0$, and hence normal in $G/H_0$, so its pullback $H$ to $N$ is finite and normal in $G$ and satisfies the proposition.
\end{proof}

\begin{proof}[Proof of \cref{thm:v.nilp.min.growth}]
Write $j := g\bigl(h(G)\bigr)$. Since $s_n(G,X)\ge1$, the theorem is trivial for $n\le2j$, so we may assume from now on that $n\ge2j$. Let $H$ and $N$ be the normal subgroups given by \cref{prop:index.bound}. It suffices to prove the result for $G/H$, so we may assume that $H=\{1\}$ and hence that $N$ is a normal nilpotent subgroup of index at most $j$ in $G$. The ball of radius $j-1$ in $G$ contains a complete set $A$ of coset representatives for $N$ \cite[Lemma 11.2.1]{tointon.book}. The set $Y:=\{axb^{-1}: a,b\in A,\,x\in X\cup X^{-1},\,axb^{-1}\in N\}$ is then a generating set for $N$ (see the proof of \cite[1.6.11]{robinson} or of \cite[Lemma 7.2.2]{hall}) and is contained in the ball of radius $2j-1$ in $G$. We therefore have
\[
s_n(G,X)\ge s_{\lfloor n/2j\rfloor}(N,Y)\ge\frac{\lfloor n/2j\rfloor^d}{2^{d^2}} 
\]
by \cref{prop:nilp.min.growth}. The fact that $n\ge2j$ implies in particular that $\lfloor n/2j\rfloor\ge n/4j$, giving the desired bound.
\end{proof}

\section{Detailed statement and proof of the main theorem}
Our main result is as follows.

\begin{theorem}\label{thm:effective.tt}
Let $C$ be the constant appearing in \cref{thm:st}, and let $d\in\N$.
Suppose $G$ is a group with finite generating set $X$ and that
\[
s_n(G, X)<\frac{n^{d}}{2^{3C^{4d}}g(C^{d})^{C^{2d}}}
\]
for some positive integer $n\ge\exp(\exp(Cd^C))$. Then $G$ has a nilpotent subgroup of index $O_{n,d}(1)$, and $\deg(G)\le d-1$, where the bound on the index is the same as the bound on the index given by \cref{thm:st}.
\end{theorem}
\begin{proof}
\cref{thm:st} implies that $G$ has a nilpotent subgroup of index $O_{n,d}(1)$, Hirsch length at most $C^d$, and growth degree $q\le C^{2d}$. \cref{thm:v.nilp.min.growth} then implies that
\[
s_m(G,X)\ge\frac{m^q}{2^{3C^{4d}}g(C^{d})^{C^{2d}}}
\]
for every $m\in\N$. Applying this with $m=n$ shows that $q<d$.
\end{proof}

\begin{proof}[Proof of \cref{cor:effective.min.growth}]
The hypothesis of \cref{thm:tt} is not satisfied for any $n<\exp(\exp(Cd^C))$ if $\eps_d$ is as stated, so \cref{thm:effective.tt} applies in every non-vacuous instance of the hypothesis.
\end{proof}

\section{Stronger bounds for nilpotent groups}
If $G$ is assumed a priori to be nilpotent, then we can improve the bounds of \cref{cor:min.growth} quite substantially. Given $d\in\N$, write
\[
f(d):=\frac1{2^{d^2}}
\]
(the constant appearing in \cref{prop:nilp.min.growth}).
\begin{prop}\label{prop:nilp.deg>d} Let $d\in\N$, and suppose that $G$ is a finitely generated nilpotent group of growth degree at least $d$ and $X$ is a finite generating set for $G$. Then $s_n(G,X) \ge f(\flr{7d/4})n^d$ for all $n\in\N$.
\end{prop}
\begin{proof} We prove the proposition by induction on $\deg(G)$. We may assume as usual that $G$ is torsion-free. We write $c$ for the class of $G$. For the induction step we assume that $\deg(G)\ge d+c$ and that the proposition has been proven for all groups of growth degree smaller than $\deg(G)$. In that case, let $x\in\gamma_c(G)$ be a non-identity element so that $N = \langle x\rangle$ is a central subgroup and $\deg(G/N) = \deg(G) - c$. The induction hypothesis then implies that $s_n(G,X) \ge s_n(G/N,XN) \ge f(\flr{7d/4})n^d$, as claimed.

It remains to prove the base cases of the induction, in which $d\le\deg(G)<d+c$. These are easy to treat on a case-by-case basis. If $d = 1$, then $G$ is infinite, so $s_n(G,X)\ge n$ and the proposition holds. We may therefore assume that $d\ge2$, so that $r(1)\ge2$ by \cref{lem:ab>2} and the class $c$ of $G$ satisfies
\begin{equation}\label{eq:c<sqrt(2r-2)}
c<\sqrt{2\deg(G)-2}
\end{equation}
by \cref{lem:class.vs.deg}. If $d=2$, then $G$ possesses a free abelian quotient of rank $2$ because $r(1)\ge2$, so the proposition holds by \cref{prop:nilp.min.growth}. The proposition holds similarly if $d=3$ and $r(1)=3$. If $d=3$ and $r(1)=2$, then $c\ge2$, so that $r(2)\ge1$ by \cref{lem:ab>2}. This implies that $\deg(G/\gamma_3(G))\ge4$, and hence that $s_n(G,X)\ge s_n(G/\gamma_3(G),X\gamma_3(G))\ge f(4)n^4$ by \cref{prop:nilp.min.growth}, and the proposition holds.

We may therefore assume that $d \ge 4$. We claim in this case that $\deg(G)\le7d/4$, which by \cref{prop:nilp.min.growth} is sufficient to prove the proposition. If $\deg(G)\le7$, then this claim is immediate.  If $\deg(G)=8$ or $9$, then \eqref{eq:c<sqrt(2r-2)} shows that $c \le 3$, and hence that $\deg(G)<7d/4$ as claimed. Finally, if $\deg(G) \ge 10$, then \eqref{eq:c<sqrt(2r-2)} implies that $c < 3\deg(G)/7$, again giving $\deg(G)<7d/4$.
\end{proof}

A similar proof establishes the following version of the above result.

\begin{prop} Given a number $\alpha > 1$, there exists an (explicitly computable) number $K = K(\alpha)$ such that if $G$ is a finitely generated nilpotent group of growth degree at least $d \ge K$ and $X$ is a finite generating set for $G$, then $s_n(G) \ge f(\flr{\alpha d})n^d$ for all $n \ge 1$.
\end{prop}
\begin{proof} Choose $K = K(\alpha)>1$ such that if $r \ge K$, then $r - \sqrt{2r-2} \ge r/\alpha$. Let $G$ be a finitely generated nilpotent group of class $c\in\N$ and growth degree at least $d \ge K$, and let $X$ be a finite generating set for $G$. We may assume as usual that $G$ is torsion-free. By the inductive argument of \cref{prop:nilp.deg>d}, we need only consider the base cases in which $\deg(G)<d+c$. Since $d>1$, \eqref{eq:c<sqrt(2r-2)} gives
$d>\deg(G)-\sqrt{2\deg(G)-2}\ge\deg(G)/\alpha$ and the claim holds by \cref{prop:nilp.min.growth}.
\end{proof}

\section{Vertex-transitive graphs}

%

In this section we prove \cref{cor:VT.min.growth.detailed}. We first provide some brief background on vertex-transitive graphs. For convenience we provide references to the third and fourth authors' paper \cite{tt.trof}, although most of what we describe is classical. See \cite{tt.trof} for more detailed background, including further references.

Let $\Gamma=(V,E)$ be a vertex-transitive graph. Given a subgroup $G\le\Aut(\Gamma)$ and a vertex $x\in V$, we write $G(x)$ for the orbit of $x$ under $G$, and $G_x$ for the stabiliser of $x$ in $G$. Note that if $G$ acts transitively on $V$, then its vertex stabilisers are all conjugate to one another; in particular, they all have the same cardinality.

Given a subgroup $H\le\Aut(\Gamma)$, we define the quotient graph $\Gamma/H$ to have vertex set $\{H(x):x\in V\}$, with $H(x)$ and $H(y)$ connected by an edge if and only if there exist $x_0\in H(x)$ and $y_0\in H(y)$ that are connected by an edge in $\Gamma$. Note in this case that $s_n(\Gamma/H)\le s_n(\Gamma)$ for all $n\in\N$. If $G$ is another subgroup of $\Aut(\Gamma)$, we say that the quotient graph $\Gamma/H$ is \emph{invariant under the action of $G$ on $\Gamma$} if for every $g\in G$ and $x\in V$, there exists $y\in V$ such that $gH(x)=H(y)$. If $H$ is normalised by $G$, then $\Gamma/H$ is invariant under the action of $G$, and the action of $G$ on $\Gamma$ descends to an action of $G$ on the vertex-transitive graph $\Gamma/H$ \cite[Lemmas 3.1 \& 3.2]{tt.trof}. When $\Gamma/H$ is invariant under $G$, we write $G_{\Gamma/H}$ for the image of $G$ in $\Aut(\Gamma/H)$ induced by this action; thus $G_{\Gamma/H}$ is the quotient of $G$ by the normal subgroup $\{g\in G:gH(x)=H(x)\text{ for every }x\in\Gamma\}$.

The automorphism group $\Aut(\Gamma)$ of the vertex-transitive graph $\Gamma$ is a topological group with the topology of pointwise convergence, which is metrisable \cite[\S4]{tt.trof}. A subset $U\subseteq\Aut(\Gamma)$ is relatively compact if and only if has a finite orbit, if and only if all its orbits are finite \cite[Lemma 4.7]{tt.trof}.

The following result allows us to study the growth of a vertex-transitive graph in terms of the growth of a closed transitive group of automorphisms.
\begin{lemma}[{\cite[Lemma 4.8]{tt.trof}}]\label{lem:growth.stabs}
Let $k\in\N$. Suppose $\Gamma$ is a connected, locally finite vertex-transitive graph and $G\le\Aut(\Gamma)$ is a closed transitive subgroup acting with vertex stabilisers of order $k$. Then there exists a finite generating set $X$ for $G$ such that $s_n(G,X)=k\cdot s_n(\Gamma)$ for all $n\in\N$.
\end{lemma}

\begin{proof}[Proof of \cref{cor:VT.min.growth.detailed}]
Let $G$ be a closed transitive subgroup of $\Aut(\Gamma)$ (for example $\Aut(\Gamma)$ itself). Since $\Gamma$ has polynomial growth, Trofimov's theorem as presented in \cite[Theorem 2.1]{tt.trof} shows that there is a compact normal subgroup $H_0\lhd G$ such that $G_{\Gamma/H_0}$ is virtually nilpotent and acts on $\Gamma/H_0$ with finite vertex stabilisers. Since orbits under $H_0$ are finite, $\Gamma/H_0$ has the same growth degree as $\Gamma$, so it suffices to prove the corollary for $\Gamma/H_0$. We may therefore assume that $H_0$ is trivial, and hence that $G$ itself is virtually nilpotent of growth degree $d$ and acts on $\Gamma$ with finite vertex stabilisers.

\cref{prop:index.bound} implies that there exist normal subgroups $H,N\lhd G$, with $H\le N$ finite and $[G:N]\le g(d)$, such that $N/H$ is torsion-free nilpotent of growth degree $d$. Write $\pi\colon G\to G_{\Gamma/H}$ for the quotient homomorphism. It is shown in \cite[Lemma 3.5]{tt.trof} that if $x$ is a vertex of $\Gamma$, then the stabiliser $(G_{\Gamma/H})_{H(x)}$ is precisely $\pi(G_x)$. In particular, $(G_{\Gamma/H})_{H(x)}$ is a homomorphic image of $G_x/(G_x\cap H)$, so that
\[
|(G_{\Gamma/H})_{H(x)}|\le[G_x:G_x\cap H].
\]
Since $N/H$ is torsion-free and $G_x$ is finite, it must be the case that $G_x\cap N\subseteq H$, and hence in particular that $G_x\cap N\subseteq G_x\cap H$. This shows that $G_x/(G_x\cap H)$ is isomorphic to a quotient of $G_x/(G_x\cap N)$, which is itself isomorphic to a subgroup of $G/N$, and so we may conclude that
\[
[G_x:G_x\cap H]\le[G:N]\le g(d).
\]
It therefore follows from \cref{thm:v.nilp.min.growth,lem:growth.stabs} that
\[
s_n(\Gamma)\ge s_n(\Gamma/H)\ge\frac1{|(G_{\Gamma/H})_{H(x)}|}\cdot\frac{n^d}{2^{d(d+2)}g(d)^d}\ge\frac{n^d}{2^{d(d+2)}g(d)^{d+1}},
\]
as required.
\end{proof}

By combining the third and fourth authors' result \cite[Corollary 1.5]{tt.trof} and \cref{cor:VT.min.growth.detailed}, one can obtain the following partially effective statement.

\begin{corollary}\label{cor:VT.min.almostEffective}
Let $d\in\N$, and suppose $\Gamma$ is a vertex-transitive graph with  degree  of growth at least $d$. Then there exists $n_0=n_0(d)\in\N$ such that
\[
s_n(\Gamma)\ge\frac{n^d}{2^{d(d+2)}g(d)^{d+1}}
\]
for every integer $n\geq n_0$.
\end{corollary}
\begin{proof}
By \cite[Corollary 1.5]{tt.trof}, there exists $n_0=n_0(d)$ such that if $s_n(\Gamma)\le n^d$ for some $n\geq n_0$, then $\deg(\Gamma)\le d$. If no such $n$ exists, then there is nothing to prove. Else, we can apply \cref{cor:VT.min.growth.detailed}.
\end{proof}

The value of $n_0=n_0(d)\in\N$ provided by the proof remains ineffective.

\appendix

\section{Universal gap in percolation}\label{appendix}
Here we sketch the details of how to explicitly bound the quantities in the proofs of Panagiotis and Severo \cite{pan-sev} to derive \cref{cor:gap}. We will not optimize our calculations; rather, we will aim for conciseness in the final result. It suffices to prove the inequalities for site percolation \cite[Proposition 7.10]{ly-per}. 

Before we consider the arguments of Panagiotis and Severo, we first consider a result that they quote from elsewhere, namely, \cite[Theorem 3.20]{hutch-toint}. The next few paragraphs are intended to be read in conjunction with \cite{hutch-toint}; all notation and terminology is as in that paper, and theorem references are also to that paper.

The proof of Theorem 3.20 shows that if $\Gamma$ is a Cayley graph of a group that is not virtually cyclic but contains a nilpotent subgroup of index at most $n\in\N$, then
there is a Cayley graph $G_1:=(V_1,E_1):=\Cay(H_0,H\cap S_0^{2n-1})$ of valency at most $(8n-4)^{2n-1}$, as well as a Cayley graph $G_2:=(V_2,E_2):=\Cay(\Gamma_0,S_0)$ of valency at most $8n-4$ that is a subgraph of $\Gamma$, such that
\[
\Prob_{1-(1-p^{1/C})^C}^{G_2,\rm{bond}}[o \leftrightarrow \infty]\ge\Prob_p^{G_1,\rm{bond}}[o \leftrightarrow \infty]\ge\Prob_p^{\Z^2,\rm{bond}}[o \leftrightarrow \infty]
\]
for all $p\in[0,1]$. Here, $C$ is the constant given by applying Lemma 2.10 with $\phi$ equal to the $(2n-1,1)$-rough embedding $G_1\to G_2$ induced by the inclusion map $H_0\to\Gamma_0$ appearing in the proof of Theorem 3.20, and $\Z^2$ has its usual Cayley graph. We will show in the next paragraph that we may take $C$ equal to $U := 2(8n-3)^{3n-2}$ 
in this case, so that
\[
\Prob_{1-(1-p^{1/U})^U}^{G_2,\rm{bond}}[o \leftrightarrow \infty]\ge\Prob_p^{\Z^2,\rm{bond}}[o \leftrightarrow \infty]
\]
for all $p\in[0,1]$. It follows from \cite[Proposition 7.11]{ly-per} that
\[
\Prob_{1-(1-p^{1/U})^{(8n-4)U}}^{G_2,\rm{site}}[o \leftrightarrow \infty]\ge\bigl(1-(1-p^{1/U})^{(8n-4)U}\bigr)\Prob_{1-(1-p^{1/U})^U}^{G_2,\rm{bond}}[o \leftrightarrow \infty].
\]
Since $G_2$ is a subgraph of $\Gamma$, we may combine the previous two displays to conclude that
\[
\Prob_{1-(1-p^{1/U})^{(8n-4)U}}^{\Gamma,\rm{site}}[o \leftrightarrow \infty]\ge\bigl(1-(1-p^{1/U})^{(8n-4)U}\bigr)\Prob_p^{\Z^2,\rm{bond}}[o \leftrightarrow \infty]
\]
for all $p\in[0,1]$, and hence
\begin{equation}\label{eq:nilp.perc}
\Prob_{1-(1-p^{1/U})^{(8n-4)U}}^{\Gamma,\rm{site}}[o \leftrightarrow \infty]\ge\bigl(1-(1-p^{1/U})^{(8n-4)U}\bigr)\Bigl(2-\frac1p\Bigr)
\end{equation}
for all $p\in[\frac12,1]$ by \cite[Theorem 1.1]{dc-tas}.

To see that we may indeed take $C=U$, and hence verify \eqref{eq:nilp.perc}, we need to bound two quantities by $U$. First, given an edge $e_1\in E_1$, we need to show that $|\Phi(e_1)|\le U$ in the notation of the proof of Lemma 2.10. To see this, note that if $x$ and $y$ are the endpoints of $e_1$, then a shortest path connecting $\phi(x)$ and $\phi(y)$ has length at most $2n$, so every edge in such a path has at least one endpoint at distance at most $n-1$ from either $\phi(x)$ or $\phi(y)$. There are at most $2(8n-3)^{n-1}$ vertices at distance at most $n-1$ from either $\phi(x)$ or $\phi(y)$, so there are at most $2(8n-3)^n$ such edges, and so $|\Phi(e_1)|\le2(8n-3)^n\le U$ as required. Second, given an edge $e_2\in E_2$, we need to show that $|\{e_1 \in E_1 \,\colon\; e_2 \in \Phi(e_1) \}|\le U$. To see this, note that if $e_2\in\Phi(e_1)$ for some $e_1\in E_1$, and if $u$ and $v$ are the endpoints of $e_2$ and $x$ and $y$ are the endpoints of $e_1$, then at least one of $\phi(x)$ and $\phi(y)$ must be within distance $n-1$ of either $u$ or $v$.
There are at most $2(8n-3)^{n-1}$ vertices at distance at most $n-1$ from either $u$ or $v$, so since $\phi$ is injective, there are at most $2(8n-3)^{n-1}(8n-4)^{2n-1}{}<{}U$ possibilities for $e_1$, as required.

The remainder of this appendix is intended to be read in conjunction with \cite{pan-sev}, and we adopt the notation of that paper except in two explicitly noted cases in the next sentence.

Replace their (3.4) by our \Cref{cor:return.prob}, which we will write as $p_n(x, y) \le \gamma_k (D/n)^{k/2}$ with $\gamma_k := 8 k^{(k+5)/2} \eps_k^{-1} \ue^{-k/2}$; here only we use our notation $\eps_k$, in which we will use (our) $C = 100$.  Although \cite{pan-sev} uses a nonlazy simple random walk, they apply such a bound only to bound the Green function, and adding laziness simply multiplies the Green function by 2, which means that we will end up with slightly larger bounds than necessary. This gives their (3.5) with $C'' = \gamma_k$ if we choose $k = 2r+2$.

In their Lemma 3.5, we have $c_n = (4n)^{-n}$ because $t_n=1/(4^nn!)>(4n)^{-n}$ for $n\ge2$.

The proof of Theorem 3.3 is broken into several cases. For the first case, we choose the same $D_0 = 2^{r^2+5}/c_{r^2+2}$ as they do and get that for $D < D_0$ and dimension at least $2r$, the inequality 
\begin{equation}\label{e.whatC}
p_n(x, y) \le C/Dn^r
\end{equation}
holds for all $n \ge 1$ when $C = \gamma_{2r} D_0^{r+1}$. In the remaining cases, $D \ge D_0$ and the dimension is at least $2r+2$. For the second case, we have \eqref{e.whatC} for all $n \ge D^r$  and $C = \gamma_{2r+2}$. For the third case, they note that $p_n(x, y) \le 1/D^{r^2+1}$ for $n \ge 1 + \int_1^{4D^{r^2+1}} \frac{16 \,\ud u}{u/\bigl(16(r^2+2)\bigr)^2} = 1 + 16^3(r^2+2)^2 \log \bigl(4D^{r^2+1}\bigr)$, so we may
set $t := 16^3 (r^2+2)^3$ to get \eqref{e.whatC} with $C = 1$ when $t \log D \le n < D^r$.
For the fourth case, we have \eqref{e.whatC} with $C = 3^r$
when $1 \le n \le 3$. For the fifth (last) case, we have \eqref{e.whatC}
when $4 \le n < t \log D$ and $C \ge \max_{3 \le D < D_0} 6(t \log D)^r/D$. Now use 
\begin{equation}\label{e.max}
\max_{u > 0} u \ue^{-u/r} = r/\ue
\end{equation}
to see that we may take $C = 6(tr/\ue)^r$. Comparing all these cases shows that in their Theorem 3.3, we may take $C_1(r) = \gamma_{2r+2}$ and $d(r) = 2r+2$.

We next turn to the proof of their Theorem 3.1. We have just seen that $d_0 = d(3) = 8$. They take $\eps := \ue^{-M}$ where $M$ is the bound in their (3.6) of the sum $t_\infty := \log 2 + C_0 \sum_{n=1}^\infty s_n$ with $s_n := (16D)^{L_n - 1} \int_{\lambda_n}^\infty \rho^n(\lambda) \,\ud\lambda$. We may take any $C_0 \ge 16/a$ with $a := \mathbb{P}[\phi_x^1 \le \lambda_1] > 1/250$, whence we may take $C_0 = 4000$. For $n = 1$, we have $s_n < 1$
because $g_1(x, x) = 1$ and $L_1 = 1$.
Now let $n \ge 2$. With $C_1 = C_1(3) = \gamma_8$, we have $g_n(x, x) = \sum_{k = L_n+1}^{L_{n+1}} p_k(x, x) \le \frac{C_1}D \sum_{k = L_n+1}^{L_{n+1}} \frac1{k^3} < \frac{C_1}{2DL_n^2}$. Thus, $\int_{\lambda_n}^\infty \rho^n(\lambda) \,\ud\lambda = \mathbb{P}[N \ge \lambda_n/\sqrt{g_n(x, x)}\,] < \sqrt{C_1/4D\pi} \cdot (n-1)^2/L_n \cdot \exp\bigl\{-DL_n^2/C_1(n-1)^4\bigr\}$ (this uses the tail bound $\mathbb{P}[N \ge \alpha] < (\sqrt{2\pi}\alpha)^{-1} \ue^{-\alpha^2/2}$).
Use \eqref{e.max} to get $D\exp\bigl\{-DL_n/C_1(n-1)^4\bigr\} \le (n-1)^4 C_1/\ue L_n$, and thus $s_n \le v_n \sqrt{C_1/108\pi} \frac{(n-1)^2}{L_n}$, where $v_u := \Bigl(\frac{16(u-1)^4 C_1}{\ue L_u}\Bigr)^{L_u}$ and $L_u := 2^{u+1}-3$ for real $u \ge 1$. Calculus shows that $\log v_n \le L_{u_*}$, where $u_*$ maximizes $v_u$ over all $u \ge 1$
(indeed, the critical point $u_*$ occurs where $0 = 2^{u+1} \log 2 \cdot \log \frac{16(u-1)^4 C_1}{\ue L_u} + L_u \bigl(\frac4{u-1} - \frac{2^{u+1} \log 2}{2^{u+1} - 3}\bigr)$, whence $\log \frac{16({u_*}-1)^4 C_1}{\ue L_{u_*}} = \bigl(1 - 3/2^{u_*+1}\bigr)\bigl(1/(1 - 3/2^{u_*+1}\bigr) - 4/(u_*-1)\bigr) < 1$).
Furthermore, we find that $u_* < 2 \log_2 C_1$
(indeed, letting $\tilde u := 2 \log_2 C_1$, we have $L_{\tilde u} = 2C_1^2 - 3$, and hence $\frac{16(\tilde u-1)^4 C_1}{\ue L_{\tilde u}} = \frac{16(2 \log_2 C_1-1)^4 C_1}{\ue (2 C_1^2 - 3)} < 1$; looking again at the derivative of $\log v_u$, we conclude that $u \mapsto \log v_u$ is decreasing at $\tilde u$, whence $u_* < \tilde u$).
It follows that $L_{u_*} < 2 C_1^2 - 3 < 2 C_1^2$, which yields $v_n < \ue^{2C_1^2}$.
Because $\sum_{n=2}^\infty (n-1)^2/\sqrt{108\pi} L_n < 1$, we find that we may take any $M \ge \log 2 + C_0\bigl(1 + \sqrt{C_1} \ue^{2C_1^2}\bigr)$. This gives that $M := \exp\bigl\{17 \exp\bigl\{10 \cdot 8^{100}\bigr\}\bigr\}$ works.

Finally, in the proof of their Theorem 1.1, we see that for dimensions at least $d_0 = d(3) = 8$, we can use $\eps_0 = \eps = \ue^{-M}$, while for smaller dimensions, we can use $\eps_0 = \eps\bigl(g(8)\bigr)$, where $\eps(n)$ is the quantity coming from Theorem 2.3 and we used \cite[Theorem 1]{bs} and our \cref{prop:index.bound}. Our \eqref{eq:g(d)} implies that $\eps\bigl(g(8)\bigr) \ge\eps(16!)$, while our \eqref{eq:nilp.perc} implies that $\eps(n)$ can be taken to be $(1-(1/2)^{1/U})^{(8n-4)U}$, where recall that $U = 2(8n-4)^{3n-2}$. The inequality $\ue^u \ge 1 + u$, valid for all real $u$, implies in particular that $1-\ue^{-u} \ge u/(1 + u)$ for all $u>-1$; applying this, we then see that $\eps(n) \ge \bigl((2U)^{-1} \log 2\bigr)^{(8n-4)U}$. Therefore, $\eps(16!)>\eps$, whence $\eps_0 = \eps$ can be used for all groups.

Now we turn to the second assertion of \cref{cor:gap}. The proof of \cite[Theorem 1.1]{pan-sev} shows that $\mathbb{P}_{1 - \eps_0}[o \leftrightarrow \infty] \ge \mathbb{P}[o \xleftrightarrow{\varphi> -1} \infty]$ when the dimension is at least $d_0$. By \cite[Proposition 2.1]{DGRSY}, we have that $\mathbb{P}[o \xleftrightarrow{\varphi> -1} \infty] \ge 1 - \exp\{D/g(o, o)\}$; that reference is in terms of a particular bond percolation, but it is easy to see that it also bounds the probability for site percolation for the  superlevel set of $\varphi$. Now $g(o,o) = \sum_{n = 1}^\infty g_n(o,o)$. Our explicit bounds above show therefore that $g(o,o) < 1 + \frac{C_1}{2D} \sum_{n=2}^\infty L_n^{-2} < C_1/25D$. Using the value above for $C_1$ implies that
\[
\mathbb{P}_{p_0}[o \leftrightarrow \infty]
>
1 - \exp\bigl\{-D^2 \exp\bigl\{-9 \exp\{100 \cdot 8^{100}\}\bigr\}\bigr\}
>
\exp\bigl\{-9 \exp\{100 \cdot 8^{100}\}\bigr\}
\]
for dimension at least $d_0$. For dimension less than $d_0$, we may again apply our \eqref{eq:nilp.perc} with $n=16!$; taking $p=2/3$, for example, yields
\[
\Prob_{1-(1-(2/3)^{1/U})^{(8n-4)U}}^{\Gamma,\rm{site}}[o \leftrightarrow \infty]\ge\frac12\left(1-(1-(2/3)^{1/U})^{(8n-4)U}\right)>\frac13.
\]
Since $(1-(2/3)^{1/U})^{(8n-4)U}\ge ((2U)^{-1} \log \frac32)^{(8n-4)U}>\eps$ by essentially the same computation as in the previous paragraph, this completes the proof.

\section*{Acknowledgements}
We are grateful to Emmanuel Breuillard, David Fisher, and Tom Hutchcroft for discussions, and an anonymous referee for comments on an earlier draft. Parts of this work were originally conducted under separate projects by non-intersecting subsets of the authors; we thank G\'abor Pete for making us aware of each other's work.

\end{document}